\documentclass[11pt]{article}
\usepackage{amsfonts,latexsym,rawfonts,amsmath,amssymb,amsthm}
\usepackage{lscape}
\usepackage[cm]{fullpage}
\usepackage{amscd, float,times,rotating}
\usepackage{pb-diagram}
\usepackage{hyperref}
\usepackage{amsthm}
\usepackage{enumerate}
\usepackage[latin1]{inputenc}
\usepackage[T1]{fontenc}
\usepackage{mathrsfs}
\usepackage{amscd}
\usepackage{amssymb}
\usepackage[english]{babel}
\usepackage{graphicx}

\usepackage{geometry}
\newtheorem{theorem}{Theorem}[section]
\newtheorem{lemma}[theorem]{Lemma}
\newtheorem{corollary}[theorem]{Corollary}

\theoremstyle{definition}
\newtheorem{definition}[theorem]{Definition}

\newtheorem{remark}[theorem]{Remark}

\numberwithin{equation}{section}

\title{The complex Monge-Amp\`{e}re equation on some compact Hermitian manifolds}

\author{Jianchun Chu}

\begin{document}

\date{}

\maketitle

\begin{abstract}
We consider the complex Monge-Amp\`{e}re equation on compact
manifolds when the background metric is a Hermitian metric (in
complex dimension two) or a kind of Hermitian metric (in higher
dimensions). We prove that the Laplacian estimate holds when $F$ is
in $W^{1,q_{0}}$ for any $q_{0}>2n$. As an application, we show
that, up to scaling, there exists a unique classical solution in
$W^{3,q_{0}}$ for the complex Monge-Amp\`{e}re equation when $F$ is
in $W^{1,q_{0}}$.
\end{abstract}

\section{Introduction}
We consider the regularity problem of the complex Monge-Amp\`{e}re
equation on some compact Hermitian manifolds. Let $(M,g)$ be a
compact Hermitian manifold of complex dimension $n\geq 2$. For a
real-valued function $F$ on $M$, we consider the Monge-Amp\`{e}re
equation
$$\det(g_{i\bar{j}}+\phi_{i\bar{j}})=e^{F}\det(g_{i\bar{j}}),$$
with $(g_{i\bar{j}}+\phi_{i\bar{j}})>0$, for a real-valued function
$\phi$ such that $\sup_{M}\phi=-1$. We write
$$\omega=\sqrt{-1}g_{i\bar{j}}dz^{i}\wedge
d\bar{z}^{j}$$ and
$$\tilde{\omega}=\sqrt{-1}\tilde{g}_{i\bar{j}}dz^{i}\wedge
d\bar{z}^{j},$$ where
$\tilde{g}_{i\bar{j}}=g_{i\bar{j}}+\phi_{i\bar{j}}$. Thus, the
Monge-Amp\`{e}re equation can be written as

\begin{equation}\label{CMAE}
\left\{ \begin{array}{ll}
\ \tilde{\omega}^{n}=e^{F}\omega^{n}\\
\ \tilde{\omega}=\omega+\sqrt{-1}\partial\bar{\partial}\phi>0\\
\ \sup_{M}\phi=-1
\end{array}.\right.
\end{equation}
We shall use the following notations, for a function $f$ and a
holomorphic coordinate $z=(z^{1},\ldots,z^{n})$,
$$f_{i\bar{j}}=\frac{\partial^{2}f}{\partial z^{i}\partial \bar{z}^{j}},~~\triangle f=g^{i\bar{j}}f_{i\bar{j}},~~\tilde{\triangle} f=\tilde{g}^{i\bar{j}}f_{i\bar{j}},$$
$$|\nabla f|^{2}=g^{i\bar{j}}f_{i}f_{\bar{j}},~~|\tilde{\nabla} f|^{2}=\tilde{g}^{i\bar{j}}f_{i}f_{\bar{j}}.$$
What is more, we use $\|f\|_{L^{p}(M,\omega)}$ and
$\|\nabla^{m}f\|_{L^{p}(M,\omega)}$ to denote the corresponding
norms with respect to $(M,\omega)$.

When $\omega$ is K\"{a}hler, the complex Monge-Amp\`{e}re equation
is very important. In the 1950s, Calabi \cite{Ca} presented his
famous conjecture and transformed that problem into (\ref{CMAE}). In
\cite{Ya}, Yau proved the existence of the classical solution of
(\ref{CMAE}) by using the continuity method and solved the Calabi's
conjecture.

The Dirichlet problem for the complex Monge-Amp\`{e}re equation is
also very important. On one hand, Bedford-Taylor \cite{BT76, BT82}
studied the weak solution. After their work, weak solution of the
complex Monge-Amp\`{e}re equation has been studied extensively.
There are many existence, uniqueness and regularity results of the
complex Monge-Amp\`{e}re equation under different conditions and we
refer the reader to \cite{Blo05, Di, DP, EGZ, GZ, Ko98, Ko08, Zh}.

On the other hand, the classical solvability of the Dirichlet
problem was established by Caffarelli-Kohn-Nirenberg-Spruck
\cite{CKNS} for strongly pseudoconvex domains in $\mathbb{C}^{n}$.
The reader can also see the work of Krylov \cite{Kry89, Kry94}. For
further information, we refer the reader to \cite{PSS} which is a
survey of some recent developments in the theory of complex
Monge-Amp\`{e}re equation.

When $\omega$ is not K\"{a}hler, the existence of the solution of the complex
Monge-Amp\`{e}re equation has been studied under some assumptions on
$\omega$ (see \cite{Ch, GL, Ha, TW1}). For a general $\omega$,
Tosatti-Weinkove \cite{TW2} has gotten the key $C^{0}$-estimate. As
an application, they have showed that, up to scaling, the complex
Monge-Amp\`{e}re equation on a compact Hermitian manifold admits a
smooth solution when the right hand side $F$ is smooth.

In \cite{CHH}, Chen-He have proved that, on a compact K\"{a}hler
manifold of complex dimension $n$, the Laplacian estimate and the
gradient estimate hold and there exists a classical solution in
$W^{3,q_{0}}$ for the complex Monge-Amp\`{e}re equation when the
right hand side $F$ is in $W^{1,q_{0}}$ for any $q_{0}>2n$.

In this paper, we generalize the work of Chen-He \cite{CHH}. We use
a different method (we don't need the gradient estimate to get the
Laplacian estimate) to consider the regularity problem of
(\ref{CMAE}) on some compact Hermitian manifolds (including compact
K\"{a}hler manifolds).

We introduce a definition first.
\begin{definition}
Let $(M,\omega)$ be a compact Hermitian manifold of complex
dimension $n$, if for any
\begin{equation*}
\phi\in\{\varphi\in
C^{2}(M)|~\omega_{\varphi}=\omega+\sqrt{-1}\partial\bar{\partial}\varphi>0,
~\|\varphi\|_{L^{\infty}(M,\omega)}\leq\Lambda_{1}~~and~-\Lambda_{2}\omega^{n}\leq
\tilde{\omega}^{n}\leq\Lambda_{2}\omega^{n}\},
\end{equation*}
there exists a constant $C=C(\Lambda_{1},\Lambda_{2},M,\omega)$,
such that
\begin{equation*}
-C\omega^{n}\leq\sqrt{-1}\partial\bar{\partial}\tilde{\omega}^{n-1}\leq
C\omega^{n}.
\end{equation*}
Then, we say $(M,\omega)$ satisfies condition $(*)$.
\end{definition}

\begin{remark}
When $n=2$, condition $(*)$ is trivial. Since
$$\partial\bar{\partial}\tilde{\omega}=\partial\bar{\partial}\omega,$$
all compact Hermitian manifolds of complex dimension 2 satisfy
condition $(*)$.
\end{remark}

\begin{remark}
When $n=3$, if $(M,\omega)$ is a compact Hermitian manifold
satisfies
$$\partial\bar{\partial}\omega=0,$$
then we have
$$\partial\bar{\partial}\tilde{\omega}^{2}=2\partial\omega\wedge\bar{\partial}\omega,$$
which implies this Hermitian manifold $(M,\omega)$ satisfies
condition $(*)$.
\end{remark}

\begin{remark}
When $n\geq4$, Condition $(*)$ is not a very strong restricted
condition. For example, if $(M,\omega)$ is a compact Hermitian
manifold satisfies
\begin{equation}\label{condition}
\partial\bar{\partial}\omega=0~~~and~~~\partial\bar{\partial}\omega^{2}=0.
\end{equation}
Then we can conclude that $\partial\bar{\partial}\omega^{k}=0$ for
all $1\leq k\leq n-1$ (see, for example, \cite{FT}), which implies
$\partial\bar{\partial}\tilde{\omega}^{k}=0$ for all $1\leq k\leq
n-1$. Thus, such Hermitian manifold (satisfying (\ref{condition}))
satisfies condition $(*)$. For example, the product of a complex
curve with a K\"{a}hler metric and a complex surface with a
non-K\"{a}hler Gauduchon metric satisfies (\ref{condition}). More
examples are constructed in \cite{FT}.
\end{remark}

\begin{remark}
All compact K\"{a}hler manifolds satisfy condition $(*)$.
\end{remark}

Now, we state our Laplacian estimate as follows.
\begin{theorem}\label{mainthm2}
Let $(M,\omega)$ be a compact Hermitian manifold of complex
dimension n. Assume that either
\begin{enumerate}[(1)]
    \item $n=2$;~or
    \item $n\geq3$ and $(M,\omega)$ satisfies condition $(*)$.
\end{enumerate}
If $\phi$ is a smooth solution of (\ref{CMAE}), then
$$\|n+\triangle\phi\|_{L^{\infty}(M,\omega)}\leq C(\|F\|_{W^{1,q_{0}}(M,\omega)},q_{0},M,\omega).$$
\end{theorem}

Usually, we need the gradient estimate to derive the Laplacian
estimate. However, the computation on Hermitian manifolds is more
complicated due to the existence of torsion terms. As a result, the
gradient estimate is very difficult to get. In order to solve this
problem, we introduce a new method to get the Laplacian estimate
directly. By using Moser's iteration (see \cite{Mo}), $L^{p}$
estimates (for example, see \cite{GT}) and some interpolation
inequalities, we can obtain the Lapalcian estimate without doing any
calculation about the gradient, which makes the argument more simple
and clear. Therefore, we believe that our ideas can be applied to
other nonlinear equations on compact manifolds.

As an application of Theorem \ref{mainthm2}, we have the following
theorem.
\begin{theorem}\label{mainthm1}
Assume that $(M,\omega)$ satisfies (1) or (2) in Theorem
\ref{mainthm2}. Let F be a function in $W^{1,q_{0}}$ for any
$q_{0}>2n$. Then, there exist a function $\phi\in W^{3,q_{0}}$ and a
constant $b$, such that
\begin{equation*} \left\{
\begin{array}{ll}
\ \tilde{\omega}^{n}=e^{F+b}\omega^{n}\\
\ \tilde{\omega}=\omega+\sqrt{-1}\partial\bar{\partial}\phi>0\\
\ \sup_{M}\phi=-1
\end{array}.\right.
\end{equation*}
\end{theorem}

~~~~~~~~~~~~~~~~\\
~~~~~~~~~~~~~~~~\\
\noindent\textbf{Acknowledgments} ~~The author would like to thank
his advisor Gang Tian for leading him to study the complex
Monge-Amp\`{e}re equation, constant encouragement and several useful
comments on an earlier version of this paper. The author would also
like to thank Valentino Tosatti for his helpful comments and
suggestions, especially for pointing out that Lemma \ref{lem2.2}
holds when the background metric is not balanced when $n\geq3$,
which helps the author to remove the assumption (balanced condition
when $n\geq3$) in an earlier version of this paper. The author would
also like to thank Wenshuai Jiang and Feng Wang for many helpful
conversations.

\section{Some preliminary computations}
We need the following $C^{0}$-estimate from \cite{TW2}.
\begin{theorem}\label{thm2.1}
For any compact Hermaitian manifold $(M,\omega)$, if $\phi$ is a
smooth solution of (\ref{CMAE}), then we have
$$\|\phi\|_{L^{\infty}(M,\omega)}\leq C,$$
where $C=C( \sup_{M}F,M,\omega)$.
\end{theorem}

We need the following lemma from \cite{TW3}.
\begin{lemma}\label{lem2.2}
Let $(M,\omega)$ be a compact Hermitian manifold of complex
dimension $n$. If $\phi$ is a smooth solution of (\ref{CMAE}), then
for any $\epsilon>0$, we have
\begin{equation}
\tilde{\triangle}(\triangle
\phi)+(\epsilon-1)\frac{|\tilde{\nabla}(\triangle\phi)|^{2}}{(n+\triangle\phi)}\geq\triangle
F-A(1+\frac{1}{\epsilon})(n+\triangle\phi)(n-\tilde{\triangle}\phi),
\end{equation}
where $A=A(M,\omega,\|F\|_{L^{\infty}(M,\omega)})$.
\end{lemma}

\begin{proof}
We need the following equation from \cite{TW3} (this equation is
(9.5) in \cite{TW3}).
\begin{equation*}
\tilde{\triangle}(\log(tr_{g}\tilde{g}))\geq\frac{2}{(tr_{g}\tilde{g})^{2}}Re(\tilde{g}^{k\bar{l}}T_{ik}^{i}(tr_{g}\tilde{g})_{\bar{l}})+\frac{\triangle
F}{tr_{g}\tilde{g}} -C_{1}tr_{\tilde{g}}g-C_{1},
\end{equation*}
where the tensor $T$ is the torsion of $(M,\omega)$ and
$C_{1}=C_{1}(M,\omega,\|F\|_{L^{\infty}(M,\omega)})$. By some
calculation, we have
\begin{equation*}
\tilde{\triangle}(\triangle
\phi)-\frac{|\tilde{\nabla}(\triangle\phi)|^{2}}{(n+\triangle\phi)}
\geq\frac{2}{(n+\triangle\phi)}Re(\tilde{g}^{k\bar{l}}T_{ik}^{i}(\triangle\phi)_{\bar{l}})+\triangle
F-C_{2}(n+\triangle\phi)(n-\tilde{\triangle}\phi),
\end{equation*}
where $C_{2}=C_{2}(M,\omega,\|F\|_{L^{\infty}(M,\omega)})$ and we
have used $tr_{\tilde{g}}g=(n-\tilde{\triangle}\phi)\geq
ne^{-\frac{F}{n}}$. By the Cauchy-Schwarz inequality, for any
$\epsilon>0$, we get
\begin{equation*}
\tilde{\triangle}(\triangle
\phi)-\frac{|\tilde{\nabla}(\triangle\phi)|^{2}}{(n+\triangle\phi)}
\geq-\epsilon\frac{|\tilde{\nabla}(\triangle\phi)|^{2}}{(n+\triangle\phi)}
-\frac{A}{\epsilon}(n+\triangle\phi)(n-\tilde{\triangle}\phi)+\triangle
F-A(n+\triangle\phi)(n-\tilde{\triangle}\phi),
\end{equation*}
where $A=A(M,\omega,\|F\|_{L^{\infty}(M,\omega)})$ and we have used
$(n+\triangle\phi)\geq ne^{\frac{F}{n}}$. Then, we complete the
proof.
\end{proof}

\begin{lemma}\label{lem2.3}
Let $(M,\omega)$ be a compact Hermitian manifold of complex
dimension $n$. If $\phi$ is a smooth solution of (\ref{CMAE}), then
for any $p\geq 1$, we have
\begin{equation*}
\begin{split}
\tilde{\triangle}(e^{f_{p}(\phi)}(n+\triangle\phi)^{p})&\geq
C_{1}(p)(n+\triangle\phi)^{p+\frac{1}{n-1}}-C_{2}(p)(n+\triangle\phi)^{p}\\[4pt]
&~~~~+pe^{f_{p}(\phi)}(n+\triangle\phi)^{p-1}\triangle F,
\end{split}
\end{equation*}
where $f_{p}(\phi)=e^{-A(p+3)\phi}$,
$C_{1}(p)=C_{1}(p,\|F\|_{L^{\infty}(M,\omega)},M,\omega)$,
$C_{2}(p)=C_{2}(p,\|F\|_{L^{\infty}(M,\omega)},M,\omega)$ and
$A=A(\|F\|_{L^{\infty}(M,\omega)},M,\omega)$ ($A$ is given in Lemma
\ref{lem2.2}).
\end{lemma}

\begin{proof}
By direct calculation, we have
\begin{equation}\label{2.2}
\begin{split}
\tilde{\triangle}(e^{f_{p}(\phi)}(n+\triangle\phi)^{p})&=f'_{p}e^{f_{p}(\phi)}(\tilde{\triangle}\phi)(n+\triangle\phi)^p+(f_{p}'^{2}+f''_{p})e^{f_{p}(\phi)}|\tilde{\nabla}\phi|^{2}(n+\triangle\phi)^{p}\\[3pt]
&~~~~+pe^{f_{p}(\phi)}\tilde{\triangle}(\triangle\phi)(n+\triangle\phi)^{p-1}+p(p-1)e^{f_{p}(\phi)}|\tilde{\nabla}(\triangle\phi)|^{2}(n+\triangle\phi)^{p-2}\\[3pt]
&~~~~+2pf'_{p}e^{f_{p}(\phi)}(n+\triangle\phi)^{p-1}Re(\tilde{g}^{k\bar{l}}\phi_{k}(\triangle\phi)_{\bar{l}}).
\end{split}
\end{equation}
By the definition of $f_{p}(\phi)$, we have
\begin{equation}\label{2.3}
\left\{ \begin{array}{l}
f'_{p}(\phi)=-A(p+3)e^{-A(p+3)\phi}<0\\[6pt]
f''_{p}(\phi)=A^{2}(p+3)^{2}e^{-A(p+3)\phi}>0
\end{array}.\right.
\end{equation}
Thus, by the Cauchy-Schwarz inequality, we have
\begin{equation*}
2Re(\tilde{g}^{k\bar{l}}\phi_{k}(\triangle\phi)_{\bar{l}})\leq
\frac{(f_{p}'^{2}+f''_{p})(n+\triangle\phi)}{-pf'_{p}}|\tilde{\nabla}\phi|^{2}+\frac{-pf'_{p}}{({f_{p}'^{2}}+f_{p}'')(n+\triangle\phi)}|\tilde{\nabla}(\triangle\phi)|^{2},
\end{equation*}
which implies
\begin{equation}\label{2.4}
\begin{split}
2pf'_{p}e^{f_{p}(\phi)}(n+\triangle\phi)^{p-1}Re(\tilde{g}^{k\bar{l}}\phi_{k}(\triangle\phi)_{\bar{l}})
&\geq-(f_{p}'^{2}+f''_{p})e^{f_{p}(\phi)}|\tilde{\nabla}\phi|^{2}(n+\triangle\phi)^{p}\\[4pt]
&~~~~-\frac{p^{2}f_{p}'^{2}}{f_{p}'^{2}+f_{p}''}e^{f_{p}(\phi)}(n+\triangle\phi)^{p-2}|\tilde{\nabla}(\triangle\phi)|^{2}.
\end{split}
\end{equation}
Combining (\ref{2.2}) and (\ref{2.4}), we have
\begin{align*}
\tilde{\triangle}(e^{f_{p}(\phi)}(n+\triangle\phi)^{p})&\geq
f_{p}'e^{f_{p}(\phi)}(n+\triangle\phi)^p\tilde{\triangle}\phi
+pe^{f_{p}(\phi)}\tilde{\triangle}(\triangle\phi)(n+\triangle\phi)^{p-1}\\[6pt]
&~~~~+|\tilde{\nabla}(\triangle\phi)|^{2}(n+\triangle\phi)^{p-2}e^{f_{p}(\phi)}\left(p(p-1)-\frac{p^{2}f_{p}'^{2}}{f_{p}'^{2}+f_{p}''}\right)\\[6pt]
&\geq pe^{f_{p}(\phi)}(n+\triangle\phi)^{p-1}\left(\tilde{\triangle}(\triangle\phi)+\left( \frac{pf''_{p}}{(f'_{p})^{2}+f''_{p}}-1\right)\frac{|\tilde{\nabla}(\triangle\phi)|^{2}}{(n+\triangle\phi)}\right)\\[6pt]
&~~~~+f_{p}'e^{f_{p}(\phi)}(n+\triangle\phi)^p\tilde{\triangle}\phi.
\end{align*}
By Lemma \ref{lem2.2} (take
$\epsilon=\frac{pf''_{p}}{(f'_{p})^{2}+f''_{p}}$), we obtain,
\begin{equation}\label{2.5}
\begin{split}
\tilde{\triangle}(e^{f_{p}(\phi)}(n+\triangle\phi)^{p})&\geq
f_{p}'e^{f_{p}(\phi)}(n+\triangle\phi)^{p}\tilde{\triangle}\phi+pe^{f_{p}(\phi)}(n+\triangle\phi)^{p-1}\triangle F \\[5pt]
&~~~~-Ape^{f_{p}(\phi)}(n+\triangle\phi)^{p}(n-\tilde{\triangle}\phi)\left(1+\frac{(f'_{p})^{2}+f''_{p}}{pf''_{p}}\right)\\[5pt]
&= nf_{p}'e^{f_{p}(\phi)}(n+\triangle\phi)^{p}+pe^{f_{p}(\phi)}(n+\triangle\phi)^{p-1}\triangle F\\[5pt]
&~~~~+e^{f_{p}(\phi)}(n+\triangle\phi)^{p}(n-\tilde{\triangle}\phi)\left(-f'_{p}-Ap\left(1+\frac{(f'_{p})^{2}+f''_{p}}{pf''_{p}}\right)\right)\\[5pt]
&\geq nf_{p}'e^{f_{p}(\phi)}(n+\triangle\phi)^{p}+pe^{f_{p}(\phi)}(n+\triangle\phi)^{p-1}\triangle F\\[6pt]
&~~~~+Ae^{f_{p}(\phi)}(n+\triangle\phi)^{p}(n-\tilde{\triangle}\phi),
\end{split}
\end{equation}
where we have used $\sup_{M}\phi=-1$ and (\ref{2.3}). It is clear
that
\begin{equation*}
   tr_{g}\tilde{g}\leq(tr_{\tilde{g}}g)^{n-1}\frac{\det\tilde{g}}{\det
   g},
\end{equation*}
which implies
\begin{equation}\label{2.6}
(n+\triangle\phi)\leq (n-\tilde{\triangle}\phi)^{n-1}e^{F}.
\end{equation}
Combining with (\ref{2.5}) and (\ref{2.6}), we complete the proof.
\end{proof}

For convenience, we introduce a notation here, we define
\begin{equation}\label{u1}
u=e^{f_{1}(\phi)}(n+\triangle\phi).
\end{equation}
Thus, by Young's inequality and Lemma \ref{lem2.3}, we have
\begin{equation}\label{u2}
\tilde{\triangle}u\geq e^{f_{1}(\phi)}\triangle F-\tilde{C},
\end{equation}
where $\tilde{C}=\tilde{C}(\|F\|_{L^{\infty}(M,\omega)},M,\omega)$.

\section{The Laplacian estimate}
In this section, we remark that our constants may differ from line
to line.
\begin{lemma}\label{lem4.1}
Let $(M,\omega)$ be a compact Hermitian manifold. If $\phi$ is a
smooth solution of (\ref{CMAE}), then for any $f\in C^{\infty}(M)$,
we have
\begin{equation*}
    |\nabla f|^{2}\leq Cu|\tilde{\nabla} f|^{2},
\end{equation*}
where $u$ is defined by (\ref{u1}) and
$C=C(\|F\|_{L^{\infty}(M,\omega)},M,\omega)$.
\end{lemma}

\begin{proof}
By direct calculation, we have
\begin{equation*}
    |\nabla f|^{2}\leq(n+\triangle\phi)|\tilde{\nabla} f|^{2}.
\end{equation*}
Combining with (\ref{u1}) and Theorem \ref{thm2.1}, we complete the
proof.
\end{proof}

\begin{lemma}\label{lem4.2}
Under the assumptions of Theorem \ref{mainthm2}, for any $p\geq0$,
we have
\begin{equation*}
\int_{M}|\nabla(u^{\frac{p}{2}})|^{2}\omega^{n}\leq
C(p^{2}+1)\int_{M} u^{p}(1+|\nabla F|^{2})\omega^{n}+Cp\int_{M}
u^{p}|\nabla\phi||\nabla F|\omega^{n}+C\int_{M} u^{p+1}\omega^{n},
\end{equation*}
where $u$ is defined by (\ref{u1}) and
$C=C(\|F\|_{L^{\infty}(M,\omega)},M,\omega)$.
\end{lemma}

\begin{proof}
By Lemma \ref{lem4.1} and direct calculation, we have

\begin{align*}
\int_{M}|\nabla(u^{\frac{p}{2}})|^{2}\omega^{n}&\leq
C_{1}\int_{M} u|\tilde{\nabla}(u^{\frac{p}{2}})|^{2}\tilde{\omega}^{n}\\
&=C_{1}p\sqrt{-1}\int_{M}\partial
u^{p}\wedge\bar{\partial}u\wedge\tilde{\omega}^{n-1}\\
&=-C_{1}p\sqrt{-1}\int_{M}
u^{p}\partial\bar{\partial}u\wedge\tilde{\omega}^{n-1}+\frac{C_{1}p}{p+1}\sqrt{-1}\int_{M}\bar{\partial}u^{p+1}\wedge\partial\tilde{\omega}^{n-1}\\
&=-C_{1}p\int_{M}
u^{p}(\tilde{\triangle}u)\tilde{\omega}^{n}-\frac{C_{1}p}{p+1}\sqrt{-1}\int_{M}
u^{p+1}\partial\bar{\partial}\tilde{\omega}^{n-1},
\end{align*}
where $C_{1}=C_{1}(\|F\|_{L^{\infty}(M,\omega)},M,\omega)$. Since
$M$ satisfies condition $(*)$ (when $n=2$, all Hermitian manifolds
satisfy the condition $(*)$), we have
\begin{equation*}
-\frac{C_{1}p}{p+1}\sqrt{-1}\int_{M}
u^{p+1}\partial\bar{\partial}\tilde{\omega}^{n-1}\leq C_{2}\int_{M}
u^{p+1}\omega^{n},
\end{equation*}
where $C_{2}=C_{2}(\|F\|_{L^{\infty}(M,\omega)},M,\omega)$. By
(\ref{u2}) and $\tilde{\omega}^{n}=e^{F}\omega^{n}$, we compute
\begin{align*}
-C_{1}p\int_{M} u^{p}(\tilde{\triangle}u)\tilde{\omega}^{n}&\leq C_{3}p\int_{M} u^{p}\left(\tilde{C}-e^{f_{1}(\phi)}\triangle F\right)\tilde{\omega}^{n}\\
&\leq C_{3}p\int_{M} u^{p}\tilde{\omega}^{n}-C_{3}p\int_{M}
e^{f_{1}(\phi)}u^{p}\left(\triangle(e^{F})-e^{F}|\nabla F|^{2}\right)\omega^{n}\\
&\leq C_{4}p\int_{M} u^{p}(1+|\nabla
F|^{2})\omega^{n}+C_{3}p\int_{M}\nabla(e^{f_{1}(\phi)}u^{p})\nabla(e^{F})\omega^{n}\\
&~~~~-\sqrt{-1}C_{3}p\int_{M} e^{f_{1}(\phi)}u^{p}\bar{\partial}
e^{F}\wedge\partial\omega^{n-1},
\end{align*}
where $C_{3}=C_{3}(\|F\|_{L^{\infty}(M,\omega)},M,\omega)$ and
$C_{4}=C_{4}(\|F\|_{L^{\infty}(M,\omega)},M,\omega)$. It is clear
that
\begin{align*}
C_{3}p\int_{M}\nabla(e^{f_{1}(\phi)}u^{p})\nabla(e^{F})\omega^{n}&=
C_{3}p\int_{M}\nabla(e^{f_{1}(\phi)})u^{p}\nabla(e^{F})\omega^{n}+C_{3}p\int_{M} e^{f_{1}(\phi)}\nabla(u^{p})\nabla(e^{F})\omega^{n}\\
&\leq C_{5}p\int_{M} u^{p}|\nabla F||\nabla
\phi|\omega^{n}+\frac{1}{2}\int_{M}|\nabla
u^{\frac{p}{2}}|^{2}\omega^{n}+C_{5}p^{2}\int_{M} u^{p}|\nabla
F|^{2}\omega^{n},
\end{align*}
where $C_{5}=C_{5}(\|F\|_{L^{\infty}(M,\omega)},M,\omega)$. Here we
have used the Cauchy-Schwarz inequality. We notice that
\begin{equation*}
-\sqrt{-1}C_{3}p\int_{M} e^{f_{1}(\phi)}u^{p}\bar{\partial}
e^{F}\wedge\partial\omega^{n-1}\leq C_{6}p\int_{M} u^{p}|\nabla
F|\omega^{n},
\end{equation*}
where $C_{6}=C_{6}(\|F\|_{L^{\infty}(M,\omega)},M,\omega)$.
Combining the above inequalities, we complete the proof.
\end{proof}

\begin{theorem}\label{thm4.3}
Under the assumptions of Theorem \ref{mainthm2}, we have
$$\|u\|_{L^{\infty}(M,\omega)}\leq C(\|u\|_{L^{\frac{q_{0}}{2}}(M,\omega)},\|F\|_{W^{1,q_{0}}(M,\omega)},q_{0},M,\omega).$$
\end{theorem}

\begin{proof}
Without loss of generality, we can assume $q_{0}<\infty$. We use the
iteration method (see \cite{Mo}). By the Sobolev inequality
(Corollary \ref{coro3.2}) and Lemma \ref{lem4.1}, for $p\geq1$, we
have
\begin{align*}
\left(\int_{M} u^{p\beta}\omega^{n}\right)^{\frac{1}{\beta}} &\leq
C_{1}\int_{M}
u^{p}\omega^{n}+C_{1}\int_{M}|\nabla(u^{\frac{p}{2}})|^{2}\omega^{n}\\
&\leq C_{1}\int_{M} u^{p}\omega^{n}+ C_{1}p^{2}\int_{M}
u^{p}(1+|\nabla F|^{2})\omega^{n}\\
&~~~~+C_{1}p\int_{M} u^{p}|\nabla\phi||\nabla
F|\omega^{n}+C_{1}\int_{M}
u^{p+1}\omega^{n}\\
&\leq C_{1}p^{2}\int_{M} u^{p+1}\omega^{n}+C_{1}p^{2}\int_{M}
u^{p}|\nabla F|^{2}\omega^{n}+C_{1}p^{2}\int_{M}
u^{p}|\nabla\phi||\nabla F|\omega^{n},
\end{align*}
where $C_{1}=C_{1}(\|F\|_{L^{\infty}(M,\omega)},M,\omega)$. Here we
have used Young's inequality and $p\leq p^{2}$. By the H\"{o}lder
inequality, we have
$$\int_{M} u^{p}|\nabla F|^{2}\omega^{n}\leq(\int_{M} u^{pr_{0}}\omega^{n})^{\frac{1}{r_{0}}}(\int_{M} |\nabla F|^{q_{0}}\omega^{n})^{\frac{2}{q_{0}}}$$
and
$$\int_{M} u^{p}|\nabla\phi||\nabla F|\omega^{n}\leq(\int_{M} u^{pr_{0}}\omega^{n})^{\frac{1}{r_{0}}}(\int_{M}|\nabla\phi|^{q_{0}}\omega^{n})^{\frac{1}{q_{0}}}(\int_{M}|\nabla F|^{q_{0}}\omega^{n})^{\frac{1}{q_{0}}},$$
where $\frac{1}{r_{0}}+\frac{2}{q_{0}}=1$. Combining the above
inequalities, when $pr_{0}\geq p+1$ (that is, $p\geq
\frac{q_{0}-2}{2})$, we obtain
\begin{align*}
\|u\|_{L^{p\beta}(M,\omega)}
&\leq\left(C_{2}p^{2}(\|\nabla\phi\|_{L^{q_{0}}(M,\omega)}+1)\right)^{\frac{1}{p}}\left(
\|u\|_{L^{p+1}(M,\omega)}^{\frac{p+1}{p}}+\|u\|_{L^{pr_{0}}(M,\omega)}\right)\\
&\leq\left(C_{2}p^{2}(\|\nabla\phi\|_{L^{q_{0}}(M,\omega)}+1)\right)^{\frac{1}{p}}\|u\|_{L^{pr_{0}}(M,\omega)}^{\frac{p+1}{p}},
\end{align*}
where $C_{2}=C_{2}(\|F\|_{W^{1,q_{0}}(M,\omega)},q_{0},M,\omega)$.
By Lemma \ref{lem3.6}, we have
\begin{equation*}
\begin{split}
\|\nabla\phi\|_{L^{q_{0}}(M,\omega)}&\leq
C_{3}\|u\|_{L^{\frac{2nq_{0}}{2n+q_{0}}}(M,\omega)}+C_{3}\\
&\leq C_{3}\|u\|_{L^{\frac{q_{0}}{2}}(M,\omega)}+C_{3},
\end{split}
\end{equation*}
where $C_{3}=C_{3}(q_{0},\|F\|_{\infty},M,\omega)$. Thus, for any
$k\geq0$, we have
\begin{equation}\label{4.1}
\|u\|_{L^{p_{k}\beta}(M,\omega)}\leq
a_{k}\|u\|_{L^{p_{k}r_{0}}(M,\omega)}^{b_{k}},
\end{equation}
where
$$a_{k}=\left(C_{4}p_{k}^{2}(\|u\|_{L^{\frac{q_{0}}{2}}(M,\omega)}+1)\right)^{\frac{1}{p_{k}}},C_{4}=C_{4}(\|F\|_{W{1,q_{0}}(M,\omega)},q_{0},M,\omega)$$
$$b_{k}=\frac{p_{k}+1}{p_{k}}~~and~~p_{k}=\frac{q_{0}-2}{2}(\frac{\beta}{r_{0}})^{k}.$$
By (\ref{4.1}), we have
\begin{equation}\label{4.2}
\|u\|_{L^{p_{k}\beta}(M,\omega)}\leq a_{k}a_{k-1}^{b_{k}}\cdots
a_{0}^{b_{k}\cdots
b_{1}}\|u\|_{L^{p_{0}r_{0}}(M,\omega)}^{b_{k}\cdots b_{0}}.
\end{equation}
Without loss of generality, we can assume that $a_{k}\geq1$, for
$k\geq0$. We observe that $\prod_{i=0}^{\infty}b_{k}$ and
$\prod_{i=0}^{\infty}a_{k}$ are convergent. In (\ref{4.2}), let
$k\rightarrow\infty$, we obtain
$$\|u\|_{L^{\infty}(M,\omega)}\leq C(\|u\|_{L^{\frac{q_{0}}{2}}(M,\omega)},\|F\|_{W^{1,q_{0}}(M,\omega)},q_{0},M,\omega).$$
\end{proof}

\begin{lemma}\label{lem4.4}
Under the assumptions of Theorem \ref{mainthm2}, for any $p\geq1$,
we have
$$\int_{M} u^{p+\frac{1}{n-1}}\omega^{n}\leq C(p)\int_{M} u^{p-1}|\nabla\phi||\nabla F|\omega^{n}+C(p)\int_{M} u^{p-1}|\nabla F|^{2}\omega^{n}+C(p),$$
where $C(p)=C(p,\|F\|_{L^{\infty}(M,\omega)},M,\omega)$.
\end{lemma}
\begin{proof}
By Lemma \ref{lem2.3}, we integrate on $(M,\tilde{\omega})$, then
for any $p\geq1$, we get
\begin{align*}
\int_{M}\tilde{\triangle}(e^{f_{p}(\phi)}(n+\triangle\phi)^{p})\tilde{\omega}^{n}&\geq
C_{1}(p)\int_{M}u^{p+\frac{1}{n-1}}\tilde{\omega}^{n}-C_{2}(p)\int_{M}
u^{p}\tilde{\omega}^{n}\\
&~~~~+p\int_{M} e^{f_{p}(\phi)}(n+\triangle\phi)^{p-1}\triangle
Fe^{F}\omega^{n},
\end{align*}
where $C_{1}(p)=C_{1}(p,\|F\|_{L^{\infty}(M,\omega)},M,\omega)$ and
$C_{2}(p)=C_{2}(p,\|F\|_{L^{\infty}(M,\omega)},M,\omega)$. Here we
have used (\ref{u1}) and Theorem \ref{thm2.1}. Since $M$ satisfies
the condition $(*)$ (when $n=2$, all Hermitian manifolds satisfy the
condition $(*)$), we have
\begin{align*}
\int_{M}\tilde{\triangle}(e^{f_{p}(\phi)}(n+\triangle\phi)^{p})\tilde{\omega}^{n}&=
n\int_{M} e^{f_{p}(\phi)}(n+\triangle\phi)^{p}\sqrt{-1}\partial\bar{\partial}\tilde{\omega}^{n-1}\\
&\leq C_{3}(p)\int_{M} u^{p}\omega^{n},
\end{align*}
where $C_{3}(p)=C_{3}(p,\|F\|_{L^{\infty}(M,\omega)},M,\omega)$.
Combining the above inequalities, we compute

\begin{align*}
\int_{M} u^{p+\frac{1}{n-1}}\omega^{n}&\leq C_{4}(p)\int_{M}
e^{f_{p}(\phi)}(n+\triangle\phi)^{p-1}\left(|\nabla
F|^{2}e^{F}-\triangle(e^{F})\right)\omega^{n}+C_{5}(p)\int_{M} u^{p}\omega^{n}\\
&\leq C_{5}(p)\int_{M} u^{p-1}|\nabla
F|^{2}\omega^{n}+C_{4}(p)\int_{M}\nabla\left(e^{f_{p}(\phi)}(n+\triangle\phi)^{p-1}\right)\nabla
Fe^{F}\omega^{n}\\
&~~~~-C_{4}(p)\sqrt{-1}\int_{M} e^{f_{p}}(n+\triangle\phi)^{p-1}\bar{\partial}e^{F}\wedge\partial\omega^{n-1}+C_{5}(p)\int_{M} u^{p}\omega^{n}\\
&\leq C_{5}(p)\int_{M} u^{p}\omega^{n}+C_{5}(p)\int_{M}
u^{p-1}|\nabla
F|^{2}\omega^{n}+C_{5}(p)\int_{M} u^{p-1}|\nabla F|\omega^{n}\\
&~~~~+C_{5}(p)\int_{M} |\nabla(u^{p-1})||\nabla
F|\omega^{n}+C_{5}(p)\int_{M} u^{p-1}|\nabla\phi||\nabla
F|\omega^{n},
\end{align*}
where $C_{4}(p)=C_{4}(p,\|F\|_{L^{\infty}(M,\omega)},M,\omega)$ and
$C_{5}(p)=C_{5}(p,\|F\|_{L^{\infty}(M,\omega)},M,\omega)$. By the
Cauchy-Schwarz inequality, we have
\begin{align*}
C_{5}(p)\int_{M}|\nabla(u^{p-1})||\nabla
F|\omega^{n}&=C_{5}(p)\int_{M}|\nabla(u^{\frac{p-1}{2}})|u^{\frac{p-1}{2}}|\nabla
F|\omega^{n}\\
&\leq C_{5}(p)\int_{M}|\nabla
(u^{\frac{p-1}{2}})|^{2}\omega^{n}+C_{5}(p)\int_{M} u^{p-1}|\nabla
F|^{2}\omega^{n}.
\end{align*}
Combining with the above inequalities and Lemma \ref{lem4.2}, we get
$$\int_{M} u^{p+\frac{1}{n-1}}\omega^{n}\leq C_{6}(p)\int_{M} u^{p}\omega^{n}+C_{6}(p)\int_{M} u^{p-1}|\nabla\phi||\nabla F|\omega^{n}+C_{6}(p)\int_{M} u^{p-1}|\nabla F|^{2}\omega^{n},$$
where $C_{6}(p)=C_{6}(p,\|F\|_{L^{\infty}(M,\omega)},M,\omega)$. By
Young's inequality, we complete the proof.
\end{proof}

Now, we are in the position to prove Theorem \ref{mainthm2}.

\begin{proof}[Proof of Theorem \ref{mainthm2}]
Without loss of generality, we assume that $q_{0}<\infty$. By Lemma
\ref{lem4.4} and $F\in W^{1,q_{0}}$, for any $p\geq1$, we have
\begin{align*}
\int_{M} u^{p+\frac{1}{n-1}}\omega^{n}&\leq C_{1}(p)\int_{M}
u^{p-1}|\nabla\phi||\nabla F|\omega^{n}+C_{1}(p)\int_{M}
u^{p-1}|\nabla
F|^{2}\omega^{n}+C_{1}(p)\\
&\leq C_{1}(p)\int_{M}
u^{p-1}|\nabla\phi|^{2}\omega^{n}+C_{2}(p)\int_{M}
u^{(p-1)\frac{q_{0}}{q_{0}-2}}\omega^{n}+C_{2}(p),
\end{align*}
where $C_{1}(p)=C_{1}(p,\|F\|_{L^{\infty}(M,\omega)},M,\omega)$,
$C_{2}(p)=C_{2}(p,\|F\|_{W^{1,q_{0}}(M,\omega)},q_{0},M,\omega)$ and
we have used the H\"{o}lder inequality in the last line. When $p$
satisfies the following condition
\begin{equation*}
p+\frac{1}{n-1}>(p-1)\frac{q_{0}}{q_{0}-2}\Leftrightarrow
p<\frac{q_{0}-2}{2n-2}+\frac{q_{0}}{2}
\end{equation*}
and $p\geq1$, we can use Young's inequality to get the following
inequality
\begin{equation*}
\int_{M} u^{p+\frac{1}{n-1}}\omega^{n}\leq C_{3}(p)\int_{M}
u^{p-1}|\nabla\phi|^{2}\omega^{n}+C_{3}(p),
\end{equation*}
where
$C_{3}(p)=C_{3}(p,\|F\|_{W^{1,q_{0}}(M,\omega)},q_{0},M,\omega)$.
Now, we take $p=\frac{q_{0}}{2}-\frac{1}{n-1}$, we obtain
\begin{equation*}
\begin{split}
\int_{M} u^{\frac{q_{0}}{2}}\omega^{n}&\leq C_{4}\int_{M}
u^{\frac{q_{0}}{2}-\beta}|\nabla\phi|^{2}\omega^{n}+C_{4}\\
&\leq\frac{1}{2}\int_{M}
u^{(\frac{q_{0}}{2}-\beta)\frac{q_{0}}{q_{0}-2\beta}}\omega^{n}+C_{4}\int_{M}|\nabla\phi|^{\frac{q_{0}}{\beta}}\omega^{n}+C_{4},
\end{split}
\end{equation*}
where $C_{4}=C_{4}(\|F\|_{W^{1,q_{0}}(M,\omega)},q_{0},M,\omega)$
and $\beta=\frac{n}{n-1}$. It then follows that
\begin{equation}\label{4.3}
\|u\|_{L^{\frac{q_{0}}{2}}(M,\omega)}\leq
C_{4}\|\nabla\phi\|_{L^{\frac{q_{0}}{\beta}}(M,\omega)}^{\frac{2}{\beta}}+C_{4}.
\end{equation}
By Lemma \ref{lem3.7}, we have
\begin{equation}\label{4.4}
\|\nabla\phi\|_{L^{\frac{q_{0}}{\beta}}(M,\omega)}\leq
C_{5}\|u\|_{L^{\frac{q_{0}}{2\beta}}(M,\omega)}^{\frac{1}{2}}+C_{5},
\end{equation}
where $C_{5}=C_{5}(q_{0},\|F\|_{L^{\infty}(M,\omega)},M,\omega)$.
Combining (\ref{4.3}), (\ref{4.4}) and $\beta>1$, we get
$$\|u\|_{L^{\frac{q_{0}}{2}}(M,\omega)}\leq C_{6}(\|F\|_{W^{1,q_{0}}(M,\omega)},q_{0},M,\omega).$$
By Theorem \ref{thm4.3}, we complete the proof.
\end{proof}

\section{The H\"{o}lder estimate of second order and solve the equation}
We note that, when $F$ is in $W^{1,q_{0}}$ for any $q_{0}>2n$,
Sobolev embedding implies that $F\in C^{\alpha_{0}}$, where
$\alpha_{0}=1-\frac{2n}{q_{0}}$. By Theorem 1.1 in \cite{TWWY}, we
have the following theorem.
\begin{theorem}\label{thm5.1}
Let $(M,\omega)$ be a compact Hermitian manifold. If $\phi$ is a
smooth solution of (\ref{CMAE}) and $F\in C^{\alpha_{0}}$, then
there exists a constant $\alpha\in(0,1)$ such that
\begin{equation*}
    \|\phi\|_{C^{2,\alpha}(M,\omega)}\leq C,
\end{equation*}
where $\alpha$ and $C$ depend only on
$\|\phi\|_{L^{\infty}(M,\omega)},\|\triangle\phi\|_{L^{\infty}(M,\omega)},\alpha_{0},\|F\|_{C^{\alpha_{0}}(M,\omega)},q_{0},M$
and $\omega$.
\end{theorem}

Now we are in the position to prove Theorem \ref{mainthm1}.
\begin{proof}[Proof of Theorem \ref{mainthm1}]
Our argument here is similar to the argument in \cite{CHH}. If $F\in
W^{1,q_{0}}$ on $M$ such that $\|F\|_{W^{1,q_{0}}(M,\omega)}\leq
\Lambda$ for some positive constant $\Lambda$. Let $\{F_{k}\}$ be a
sequence of smooth functions such that $F_{k}\rightarrow F$ in
$W^{1,q_{0}}$. In particular, we can assume
$\|F_{k}\|_{W^{1,q_{0}}(M,\omega)}\leq \Lambda+1$ for any $k$. By
\cite{TW2}, there is a unique smooth solution $\phi_{k}$ and
constant $b_{k}$ such that
\begin{equation*}
\det(g_{i\bar{j}}+(\phi_{k})_{i\bar{j}})=e^{F+b_{k}}\det(g_{i\bar{j}}),
\end{equation*}
such that $(g_{i\bar{j}}+(\phi_{k})_{i\bar{j}})>0$ with normalized
condition $\sup_{M}\phi_{k}=-1$. By Maximum Principle, we have
\begin{equation}\label{5.1}
|b_{k}|\leq C_{1}(\|F_{k}\|_{L^{\infty}(M,\omega)},M,\omega).
\end{equation}
By Theorem \ref{mainthm2}, Theorem \ref{thm2.1} and Theorem
\ref{thm5.1}, there exists a constant $\alpha\in(0,1)$ such that
\begin{equation*}
\|\phi_{k}\|_{C^{2,\alpha}(M,\omega)}\leq
C_{2}(\|F_{k}\|_{W^{1,q_{0}}(M,\omega)},q_{0},M,\omega).
\end{equation*}
To get $W^{3,q_{0}}$-estimate, we can localize the estimate as
follows. Let $\partial$ denote an arbitrary first order differential
operator in a domain $\Omega\subset M$. Since we have
$C^{2,\alpha}$-estimate, we compute in $\Omega$
\begin{equation*}
\tilde{\triangle}_{g_{k}}(\partial\phi_{k})=\partial(F_{k}+\log(\det(g_{i\bar{j}})))-(g_{k})^{i\bar{j}}\partial
g_{i\bar{j}},
\end{equation*}
where $(g_{k})_{i\bar{j}}=g_{i\bar{j}}+(\phi_{k})_{i\bar{j}}$. Since
$\tilde{\triangle}_{g_{k}}$ is a uniform elliptic operator, by
$L^{p}$ estimates (for example, see \cite{GT}), for any
$\Omega'\subset\Omega$, we have
\begin{equation*}
\|\partial\phi_{k}\|_{W^{2,q_{0}}(\Omega',\omega)}\leq
C_{3}(\Omega,\Omega',q_{0},\Lambda,\omega),
\end{equation*}
which implies
\begin{equation}\label{5.2}
\|\phi_{k}\|_{W^{3,q_{0}}(M,\omega)}\leq
C_{4}(\|F\|_{W^{1,q_{0}}(M,\omega)},q_{0},\Lambda,M,\omega).
\end{equation}
By (\ref{5.1}) and (\ref{5.2}), we know that there is a subsequence
$\{(\phi_{k_{l}},b_{k_{l}})\}$ of $\{(\phi_{k},b_{k})\}$ such that
$\{b_{k_{l}}\}$ converges to $b$ and $\{\phi_{k_{l}}\}$ converges to
$\phi\in W^{3,q_{0}}$ such that $(g_{i\bar{j}}+\phi_{i\bar{j}})>0$,
which defines a $W^{1,q_{0}}$ Hermitian metric. Hence $\phi$ with
constant $b$ is a classical solution of the complex Monge-Amp\`{e}re
equation. The uniqueness follows from Remark 5.1 in \cite{TW1}.
\end{proof}

\section{Appendix}
Let $g_{\mathbb{R}}$ denote the Riemannian metric induced by $g$,
thus $(M,g_{\mathbb{R}})$ is a Riemannian manifold of real dimension
$2n$. In Appendix, we deduce some interpolation inequalities on
Hermitian manifold $(M,\omega)$ by using some fundamental
inequalities on Riemannain manifold $(M,g_{\mathbb{R}})$.

Let us recall the definition of $g_{\mathbb{R}}$ first. For any
local holomorphic coordinates $(z^{1},\cdots,z^{n})$ with
$z^{i}=x^{i}+\sqrt{-1}y^{i}$,
$(x^{1},\cdots,x^{n},y^{1},\cdots,y^{n})$ form a smooth local
coordinates. We define,
\begin{equation*}
g_{\mathbb{R}}(\frac{\partial}{\partial
x^{i}},\frac{\partial}{\partial x^{j}})
=g_{\mathbb{R}}(\frac{\partial}{\partial
y^{i}},\frac{\partial}{\partial y^{j}}) =2\Re(g_{i\bar{j}})
\end{equation*}
while
\begin{equation*}
g_{\mathbb{R}}(\frac{\partial}{\partial
x^{i}},\frac{\partial}{\partial y^{j}}) =2\Im(g_{i\bar{j}}).
\end{equation*}
For Riemannian metric $g_{\mathbb{R}}$, let $\nabla_{\mathbb{R}}$
and $dV_{\mathbb{R}}$ denote the Levi-Civita connection and the
volume form, respectively. By direct calculation, we have
\begin{equation}\label{3.1}
dV_{\mathbb{R}}=\frac{1}{n!}\omega^{n}.
\end{equation}
For convenience, we introduce some notations. For any function $f\in
C^{\infty}(M)$, let $\nabla^{m}_{\mathbb{R}}f$ and
$\triangle_{\mathbb{R}}f$ denote the $m^{th}$ covariant derivative
and the Laplacian of $f$ with respect to $g_{\mathbb{R}}$. Let
$\|f\|_{L^{p}(M,g_{\mathbb{R}})}$ and
$\|\nabla^{m}_{\mathbb{R}}f\|_{L^{p}(M,g_{\mathbb{R}})}$ denote the
corresponding norms with respect to $(M,g_{\mathbb{R}})$.

Thus, by (\ref{3.1}) and some calculations, we have the following
lemma.
\begin{lemma}\label{lem3.1}
For any $f\in C^{\infty}(M)$, we have
\begin{equation*}
\|f\|_{L^{p}(M,g_{\mathbb{R}})}=C_{1}(p)\|f\|_{L^{p}(M,\omega)},~~~\|\nabla_{\mathbb{R}}f\|_{L^{p}(M,g_{\mathbb{R}})}=C_{2}(p)\|\nabla
f\|_{L^{p}(M,\omega)},
\end{equation*}
where $C_{1}(p)=C_{1}(p,n)$ and $C_{2}(p)=C_{2}(p,n)$.
\end{lemma}

\begin{corollary}\label{coro3.2}
For any $f\in C^{\infty}(M)$, we have Sobolev inequality
\begin{equation*}
\left(\int_{M} f^{2\beta}\omega^{n}\right)^{\frac{1}{\beta}}\leq
C\int_{M} f^{2}\omega^{n}+C\int_{M}|\nabla f|^{2}\omega^{n},
\end{equation*}
where $\beta=\frac{n}{n-1}$ and $C=C(M,\omega)$.
\end{corollary}

\begin{proof}
By Sobolev embedding $W^{1,2}(M,g_{\mathbb{R}})\hookrightarrow
L^{2\beta}(M,g_{\mathbb{R}})$, we have
\begin{equation*}
\left(\int_{M}
f^{2\beta}dV_{\mathbb{R}}\right)^{\frac{1}{\beta}}\leq C_{s}\int_{M}
f^{2}dV_{\mathbb{R}}+C_{s}\int_{M}|\nabla_{\mathbb{R}}
f|^{2}dV_{\mathbb{R}},
\end{equation*}
where $C_{s}=C_{s}(M,g_{\mathbb{R}})$. Thus, combining with Lemma
\ref{lem3.1}, we complete the proof.
\end{proof}

Because $(M,g_{\mathbb{R}})$ is a Riemannian manifold of real
dimension $2n$. So we have the following interpolation inequality
(for example, see \cite{Au}).

\begin{theorem}\label{thm3.3}
Let $q$, $r$ be real numbers $1\leq q,r\leq+\infty$ and $j$, $m$
integers $0\leq j<m$. Then there exists a constant
$$C=C(M,g_{\mathbb{R}},m,j,q,r,\alpha)$$
such that for all $f\in C^{\infty}(M)$ with $\int_{M}
fdV_{\mathbb{R}}=0$, we have
\begin{equation}\label{3.2}
\|\nabla^{j}_{\mathbb{R}}f\|_{L^{p}(M,g_{\mathbb{R}})}\leq C
\|\nabla^{m}f\|_{L^{r}(M,g_{\mathbb{R}})}^{\alpha}
\|f\|_{L^{q}(M,g_{\mathbb{R}})}^{1-\alpha},
\end{equation}
where
$$\frac{1}{p}=\frac{j}{2n}+\alpha(\frac{1}{r}-\frac{m}{2n})+(1-\alpha)\frac{1}{q}$$
for all $\alpha$ in the interval $\frac{j}{m}\leq\alpha\leq 1$, for
which $p$ is non-negative. If $r=\frac{2n}{m-j}\neq 1$, then
(\ref{3.2}) is not valid for $\alpha=1$.
\end{theorem}

\begin{corollary}\label{coro3.4}
Let $f\in C^{\infty}(M)$, for any $\epsilon>0$ and $1\leq p<\infty$,
we have
$$\|\nabla_{\mathbb{R}} f\|_{L^{p}(M,g_{\mathbb{R}})}\leq \epsilon \|\nabla^{2}_{\mathbb{R}}f\|_{L^{p}(M,g_{\mathbb{R}})}+C(\epsilon,p)\|f\|_{L^{p}(M,g_{\mathbb{R}})},$$
where $C(\epsilon,p)=C(\epsilon,p,M,\omega)$.
\end{corollary}

\begin{proof}
Define $\tilde{f}=f-\frac{1}{Vol(M,g_{\mathbb{R}})}\int_{M}
fdV_{\mathbb{R}}$, then $\int_{M} \tilde{f}dV_{\mathbb{R}}=0$. By
Theorem \ref{thm3.3}, we have
$$\|\nabla_{\mathbb{R}}\tilde{f}\|_{L^{p}(M,g_{\mathbb{R}})}\leq C_{1}(p)\|\nabla^{2}_{\mathbb{R}}\tilde{f}\|_{L^{p}(M,g_{\mathbb{R}})}^{\frac{1}{2}}\|\tilde{f}\|_{L^{p}(M,g_{\mathbb{R}})}^{\frac{1}{2}},$$
where $C_{1}(p)=C_{1}(p,M,g_{\mathbb{R}})$. Thus, by the
Cauchy-Schwarz inequality, for any $\epsilon>0$, we obtain
$$\|\nabla_{\mathbb{R}}\tilde{f}\|_{L^{p}(M,g_{\mathbb{R}})}\leq \epsilon \|\nabla^{2}_{\mathbb{R}}\tilde{f}\|_{L^{p}(M,g_{\mathbb{R}})}+C_{2}(\epsilon,p)\|\tilde{f}\|_{L^{p}(M,g_{\mathbb{R}})},$$
where $C_{2}(\epsilon,p)=C_{2}(\epsilon,p,M,g_{\mathbb{R}})$. By the
definition of $\tilde{f}$, we complete the proof.
\end{proof}

\begin{lemma}\label{lem3.5}
Let $(M,\omega)$ be a compact Hermitian manifold of complex
dimension $n$. If $\phi$ is a smooth solution of (\ref{CMAE}), then
for any $1<p<\infty$, we have
\begin{equation*}
\|\triangle_{\mathbb{R}}\phi\|_{L^{p}(M,\omega)}\leq
C_{1}(p)\|\triangle\phi\|_{L^{p}(M,\omega)}+C_{2}(p),
\end{equation*}
where $C_{1}=C_{1}(p,n)$ and
$C_{2}(p)=C_{2}(p,\|F\|_{L^{\infty}(M,\omega)},M,\omega)$.
\end{lemma}

\begin{proof}
By some calculations, we have
\begin{equation}\label{3.3}
\|\triangle_{\mathbb{R}}\phi\|_{L^{p}(M,g_{\mathbb{R}})}\leq2\|\triangle\phi\|_{L^{p}(M,g_{\mathbb{R}})}+C_{3}(p)\|\nabla_{\mathbb{R}}\phi\|_{L^{p}(M,g_{\mathbb{R}})},
\end{equation}
where $C_{3}=C_{3}(p,M,\omega)$. For (\ref{3.3}), one can find more
details in \cite{To} (Lemma 3.2 in \cite{To} shows the exact
relation between $\triangle_{\mathbb{R}}$ and $2\triangle$). By
Corollary \ref{coro3.4}, we obtain
\begin{equation}\label{3.4}
C_{3}(p)\|\nabla_{\mathbb{R}}\phi\|_{L^{p}(M,g_{\mathbb{R}})}\leq\frac{1}{2}\|\triangle_{\mathbb{R}}\phi\|_{L^{p}(M,g_{\mathbb{R}})}+C_{4}(p)\|\phi\|_{L^{p}(M,g_{\mathbb{R}})},
\end{equation}
where $C_{4}=C_{4}(p,M,\omega)$. Combining with (\ref{3.3}) and
(\ref{3.4}), we obtain
\begin{equation*}
\|\triangle_{\mathbb{R}}\phi\|_{L^{p}(M,g_{\mathbb{R}})}\leq4\|\triangle\phi\|_{L^{p}(M,g_{\mathbb{R}})}+C_{5}(p)\|\phi\|_{L^{p}(M,g_{\mathbb{R}})},
\end{equation*}
where $C_{5}=C_{5}(p,M,\omega)$. By Theorem \ref{thm2.1} and Lemma
\ref{lem3.1}, we complete the proof.
\end{proof}

\begin{lemma}\label{lem3.6}
Under the assumptions of Theorem \ref{mainthm2}, for any $1<p<2n$,
we have
\begin{equation*}
\|\nabla\phi\|_{L^{\frac{2np}{2n-p}}(M,\omega)}\leq
C(p)\|u\|_{L^{p}(M,\omega)}+C(p),
\end{equation*}
where $u$ is defined by (\ref{u1}) and
$C(p)=C(p,\|F\|_{L^{\infty}(M,\omega)},M,\omega)$.
\end{lemma}

\begin{proof} By Sobolev embedding
$W^{2,p}(M,g_{\mathbb{R}})\hookrightarrow
W^{1,\frac{2np}{2n-p}}(M,g_{\mathbb{R}})$, we have
\begin{equation*}
\|\nabla_{\mathbb{R}}\phi\|_{L^{\frac{2np}{2n-p}}(M,g_{\mathbb{R}})}\leq
C_{1}(p)\|\nabla^{2}_{\mathbb{R}}\phi\|_{L^{p}(M,g_{\mathbb{R}})}+C_{1}(p)\|\nabla_{\mathbb{R}}\phi\|_{L^{p}(M,g_{\mathbb{R}})}+C_{1}(p)\|\phi\|_{L^{p}(M,g_{\mathbb{R}})},
\end{equation*}
where $C_{1}(p)=C_{1}(p,M,g_{\mathbb{R}})$. Combining with Corollary
\ref{coro3.4}, we have
\begin{equation*}
\|\nabla\phi\|_{L^{\frac{2np}{2n-p}}(M,g_{\mathbb{R}})}\leq
C_{2}(p)\|\nabla^{2}_{\mathbb{R}}\phi\|_{L^{p}(M,g_{\mathbb{R}})}+C_{2}(p)\|\phi\|_{L^{p}(M,g_{\mathbb{R}})},
\end{equation*}
where $C_{2}(p)=C_{2}(p,M,g_{\mathbb{R}})$. By Theorem \ref{thm2.1}
and $L^{p}$ estimates (for example, see \cite{GT}), we have
\begin{equation*}
\|\nabla\phi\|_{L^{\frac{2np}{2n-p}}(M,g_{\mathbb{R}})}\leq
C_{3}(p)\|\triangle_{\mathbb{R}}\phi\|_{L^{p}(M,g_{\mathbb{R}})}+C_{3}(p),
\end{equation*}
where
$C_{3}(p)=C_{3}(p,\|F\|_{L^{\infty}(M,\omega)},M,g_{\mathbb{R}})$.
By Lemma \ref{lem3.1} and Lemma \ref{lem3.5}, we have
\begin{equation*}
\|\nabla\phi\|_{L^{\frac{2np}{2n-p}}(M,\omega)}\leq
C_{4}(p)\|\triangle\phi\|_{L^{p}(M,\omega)}+C_{4}(p),
\end{equation*}
where
$C_{4}(p)=C_{4}(p,\|F\|_{L^{\infty}(M,\omega)},M,g_{\mathbb{R}})$.
By (\ref{u1}) and Theorem \ref{thm2.1}, we complete the proof.
\end{proof}

\begin{lemma}\label{lem3.7}
Let $p$, $r$ be real numbers $1<p,r<\infty$. Under the assumptions
of Theorem \ref{mainthm2}, we have
\begin{equation*}
\|\nabla\phi\|_{L^{p}(M,\omega)}\leq
C(p,r)\|u\|_{L^{r}}^{\alpha}+C(p,r),
\end{equation*}
where $C(p,r)=C(p,r,\|F\|_{L^{\infty}(M,\omega)},M,\omega)$ and
\begin{equation*}
\frac{1}{p}=\frac{1}{2n}+\alpha(\frac{1}{r}-\frac{1}{n}),
\end{equation*}
for $\alpha$ in the $\frac{1}{2}\leq\alpha<1$.
\end{lemma}

\begin{proof}
Define $\tilde{\phi}=\phi-\frac{1}{Vol(M,g_{R})}\int_{M}\phi
dV_{\mathbb{R}}$, then $\int_{M}\tilde{\phi}dV_{\mathbb{R}}=0$, by
Theorem \ref{thm2.1}, Lemma \ref{lem3.1} and Theorem \ref{thm3.3},
we have
\begin{equation*}
\|\nabla_{\mathbb{R}}\tilde{\phi}\|_{L^{p}(M,g_{\mathbb{R}})}\leq
C_{1}(p,r)\|\nabla^{2}_{\mathbb{R}}\tilde{\phi}\|_{L^{r}(M,g_{\mathbb{R}})}^{\alpha}
\end{equation*}
which implies
\begin{equation*}
\|\nabla_{\mathbb{R}}\phi\|_{L^{p}(M,g_{\mathbb{R}})}\leq
C_{1}(p,r)\|\nabla^{2}_{\mathbb{R}}\phi\|_{L^{r}(M,g_{\mathbb{R}})}^{\alpha},
\end{equation*}
where $C_{1}(p,r)=C_{1}(p,r,\|F\|_{L^{\infty}(M,\omega)},M,\omega)$
and $\alpha=\frac{(2n-p)r}{(2n-2r)p}$. Combining Lemma \ref{lem3.1},
Lemma \ref{lem3.5}, (\ref{u1}) and $L^{p}$ estimates (for example,
see \cite{GT}), we complete the proof.
\end{proof}

~~~~~~
\\

\noindent{Jianchun Chu}\\
School of Mathematical Sciences, Peking University\\
Yiheyuan Road 5, Beijing, 100871, China\\
Email:{chujianchun@pku.edu.cn}

\end{document}